\documentclass[12pt]{amsart}
\usepackage{amsmath, amsthm, amssymb, amsfonts, enumerate}
\usepackage[colorlinks=true,linkcolor=blue,urlcolor=blue]{hyperref}
\IfFileExists{dsfont.sty}{\usepackage{dsfont}}{}
\usepackage{color}
\usepackage{geometry}
\IfFileExists{todonotes.sty}{\usepackage{todonotes}}{}
\usepackage{graphicx}
\usepackage{caption}
\usepackage{float}
\usepackage{mathtools}
\usepackage{mathrsfs}
\IfFileExists{soul.sty}{\usepackage{soul}}{}
\mathtoolsset{showonlyrefs}
\numberwithin{equation}{section}
\def \P{\mathsf P}
\def \E{\mathsf E}
\def \R{\mathbb R}
\def \ud{\mathrm{d}}
\def \e{\mathrm{e}}
\newcommand{\ind}{\mathbf{1}}
\newtheorem{theorem}{Theorem}[section]
\newtheorem{lemma}{Lemma}[section]
\newtheorem{corollary}{Corollary}[section]
\newtheorem{proposition}{Proposition}[section]
\newtheorem{example}{Example}[section]
\newtheorem{remark}{Remark}[section]
\DeclareMathOperator{\Var}{Var}
\DeclareMathOperator{\Cov}{Cov}

\newcommand{\cx}{\mathrm{cx}}

\title{Convex order and preservation of convexity for Bayesian posterior updates}
\author[Bayraktar]{Erhan Bayraktar}
\thanks{E. Bayraktar acknowledges support from the National Science Foundation under Grant No.~DMS-2602036 and from the Susan M. Smith Professorship.}
\author[Wang]{Yuqiong Wang}
\address{Erhan Bayraktar: Department of Mathematics, University of Michigan, Ann Arbor, MI 48109, USA}
\email{erhan@umich.edu}
\address{Yuqiong Wang: Department of Mathematics, University of Michigan, Ann Arbor, MI 48109, USA}
\email{yuqw@umich.edu}
\date{\today}

\begin{document}

\begin{abstract}
We study how the response of a Bayesian posterior statistic to future observations changes as information accumulates. For a non-decreasing function $T$, define $\Pi_n^T=\E[T(\Theta)\vert \mathcal F_n]$, where $\Theta$ has an arbitrary prior and the observations come from a one-parameter exponential family. Conditioning on the same current value of $\Pi^T$, we show that the posterior statistic after additional observations is larger in convex order when the current posterior is based on fewer observations. We also prove preservation of convexity: the expected value of a convex function of the future posterior statistic is convex in the current posterior statistic. Together, these two properties provide structural tools for establishing time-monotonicity results in dynamic Bayesian decision and optimal stopping problems. If the exponential family contains an infinitely divisible distribution, the results extend to a continuous-time observation model through a family of L\'evy processes.
\end{abstract}

\maketitle

\section{Introduction}\label{sec_intro}
Let $X_1,X_2,\ldots$ be conditionally independent observations whose distribution depends on an unknown parameter $\Theta$. We assume that their conditional distributions belong to a one-parameter exponential family and that $\Theta$ has prior distribution $\mu$. For a non-decreasing function $T$, define
\[\Pi_n^T:=\E[T(\Theta)\vert\mathcal F_n],\]
where $\mathcal F_n:=\sigma(X_1,\ldots,X_n)$. Different choices of $T$ correspond to different applications and posterior statistics. The choice $T(u)=u$ gives the posterior mean, while $T(u)=\ind_{\{u>\theta_0\}}$ gives the posterior probability of the composite hypothesis $\Theta>\theta_0$. Another example is $T(u)=B'(u)$, for which $\Pi_n^T$ is the posterior predictive mean of the next observation.

Additional observations make posterior statistics more dispersed in convex order. Indeed, the tower property gives
\[\Pi_m^T\leq_{\cx}\Pi_n^T,\quad m<n.\]
We ask a different question: how does the future distribution of $\Pi^T$ depend on the amount of information already accumulated? Suppose that two observers report the same current value $z$ of $\Pi^T$, but one has made $n$ observations and the other has made $m<n$ observations. Since $\Pi^T$ is a martingale, both future posterior statistics have mean $z$. Which observer faces the more dispersed posterior update after the same block of additional observations?

We show that it is the observer with the shorter history. Writing $K_{n,k}^T(z,\cdot)$ for the $k$-step posterior transition law from state $z$ at time $n$, and $\mathcal K_{n,k}^Tf(z):=\int f(w)K_{n,k}^T(z,\ud w)$ for its operator, we prove that, for $m<n$ and every fixed $k\geq 1$,
\begin{equation}\label{main}
K_{m,k}^T(z,\cdot)\geq_{\cx}K_{n,k}^T(z,\cdot).
\end{equation}
The kernel is defined on the admissible posterior-state domain and therefore remains meaningful when $\{\Pi_n^T=z\}$ has probability zero. Thus, at the same value of the posterior statistic, the earlier transition law is a mean-preserving spread of the later one.

The result is related to the comparison of statistical experiments and the value of information. Blackwell~\cite{BL} compares experiments under a fixed prior through the dispersion of the posterior beliefs they generate. Our setting is different: the same future experiment is applied to two current posteriors generated by histories of different lengths. DeGroot~\cite{DG} measures information through the expected reduction from prior to posterior uncertainty. We instead compare the value of the same future experiment at different information times, conditional on a matched scalar posterior state.

Bikhchandani and Mamer~\cite{BM} study the decreasing marginal value of conditionally i.i.d.\ signals under quadratic loss. They establish an ex ante result for several conjugate exponential-family models. For their ex post comparison, they present a normal--normal model---the only example of this property they report finding---in which the value of an additional signal decreases after every realized signal.

Our result is related, but the comparison is different. To make the distinction precise, let $T(\theta)=\theta$, let $v(z)=z^2$, and define
\[I_t(z):=\mathcal K_{t,1}^Tv(z)-v(z).\]
Bikhchandani and Mamer compare $I_{n-1}(z)$ with $I_n(z')$, where $z'$ is the posterior state reached after the next signal. Their comparison therefore follows the realized evolution of the posterior state. We instead compare $I_{n-1}(z)$ with $I_n(z)$, holding the posterior state fixed while varying the amount of information already accumulated. Under this matched-state comparison, our result applies to one-parameter exponential-family observation models and arbitrary priors satisfying our assumptions.

Related monotonicity results also appear in Bayesian bandit models. For conjugate models in exponential families, holding the prior mean reward fixed, Yu~\cite{YY} shows that the optimal bandit value decreases with the prior sample size. More generally, stochastic comparison arguments often combine an ordering of transition kernels with preservation of a class of test functions, such as convex functions; see \cite{BR,RL}. In our setting, both the ordering of posterior kernels and preservation of convexity are proved directly from the Bayesian posterior structure.

Our initial motivation comes from the composite testing problem of Ekstr\"om and Wang~\cite{EW}, where the posterior process is $\Pi_n=\P(\Theta>\theta_0\vert\mathcal F_n)$. They study the sequential problem of testing whether the unknown parameter exceeds a given threshold. Their value function is
\[V(n,\pi)=\inf_{\tau\in\mathcal T}\E_{n,\pi}\left[\Pi_{n+\tau}\wedge(1-\Pi_{n+\tau})+c\tau\right].\]
They compare posterior distributions at different information times while holding the posterior probability $\pi$ fixed, and show that the posterior distribution of $\Theta$ becomes more concentrated around the threshold $\theta_0$. This suggests the dynamic posterior convex-order comparison. To prove time monotonicity of the value function, however, one must compare the law of the next posterior probability, not only the current posterior. They therefore impose the one-step version of \eqref{main} as Assumption~5.1 and show in their Theorem~5.2 that it implies monotonicity of the value function and hence of the stopping boundaries. They verify this time monotonicity for a large class of model specifications and formulate the general exponential-family case as a conjecture. Our first main result, Theorem~\ref{thm1}, establishes this convex-order comparison and extends it from posterior probabilities to arbitrary non-decreasing posterior functionals $T$, and from one step to every fixed horizon $k$.

A related application is the quickest-search problem of Bayraktar and Kravitz~\cite{BKsearch}, in which each Brownian channel has a fixed binary type and the posterior probability is a martingale. In their binary model, the posterior probability determines the entire posterior distribution, so the within-channel transition kernel is time-homogeneous. With a composite prior over channel quality, accumulated information becomes an additional state variable. The results of this paper describe how the within-channel continuation kernel varies with this state: it decreases in convex order with accumulated information and preserves convexity. They therefore provide the structural ingredients for extending quickest-search models to composite channel types and more general exponential-family observations.

The distinction between the current and future posteriors is important. A static comparison concerns the current conditional law of $\Theta$. Our comparison concerns the random posterior state obtained after additional observations, which is the object that appears in dynamic decision problems through continuation values of the form $\mathcal K_{n,k}^Tf(z)$. Neither comparison implies the other. Sequential testing gives a simple illustration: along a fixed $z$-level curve, $\ind_{\{\Theta>\theta_0\}}$ has the same Bernoulli$(z)$ distribution at every time, while the future posterior probabilities have ordered transition laws.

For dynamic decision problems, however, the convex-order comparison alone is not enough. To iterate the Bellman equation, the value function must remain convex or concave after posterior updating. Our second main result, Theorem~\ref{thm_convex}, establishes precisely this preservation property: if the payoff is convex as a function of the posterior state, then its expected value after future observations is again convex in the current posterior state. This closure property allows the convex-order comparison to be reapplied at each backward step of the dynamic program and yields time-monotonicity results for stopping problems. Related preservation results for Markov and parabolic operators include \cite{JT2,JT,ET,ET2}; for a discrete-time example, see \cite{EW}. Our result establishes preservation directly for posterior transition kernels in one-parameter exponential families.

Our convexity-preservation result differs from classical convexity results for partially observed Markov decision processes (POMDPs), such as \cite{SmSo}. Those results establish convexity on the full belief space, whereas our convexity is with respect to the scalar posterior statistic $z$ and does not follow directly from belief-space convexity. Our proof also does not use the total-positivity methods of \cite{K_TP}.

Instead, both main results come from the same lemma applied to two different perturbations of the posterior family. Along a fixed-$z$ level curve, the time derivative of the posterior is the signed measure $G_t\mu_t$; at fixed time, the second derivative with respect to the state variable is $Q_z\mu_{t,z}$. The function $G_t$ is concave, whereas $Q_z$ is convex. These perturbations satisfy orthogonality relations of the same form, but for slightly different reasons. Normalization gives the first relation in each pair. Along the time level curve, $\int_ST \ud\mu_t=z$ is constant, while along the state parametrization, $\int_ST \ud\mu_{t,z}=z$ is affine in $z$. Differentiating once in the first case and twice in the second therefore gives
\[\int_SG_t \ud\mu_t=\int_STG_t\ud\mu_t=0,\quad\int_SQ_z\ud\mu_{t,z}=\int_STQ_z \ud\mu_{t,z}=0.
\]
The other ingredient has the same form in both proofs. Freeze the upper event $A$ selected by a call payoff and set
\[q(u):=\P_u(Y\in A)\]
for the future experiment $Y$. The monotone-likelihood-ratio property makes $q$ non-decreasing. For the threshold $a$ associated with $A$, Lemma~\ref{lem_frozen} then gives
\[\int_S(T(u)-a)q(u)G_t(u)\mu_t(\ud u)\leq0,\quad\int_S(T(u)-a)q(u)Q_z(u)\mu_{t,z}(\ud u)\geq0.\]
The signs are opposite precisely because $G_t$ is concave while $Q_z$ is convex. The first sign makes the frozen-event functional non-increasing in information time; the resulting envelope comparison yields the convex-order theorem. The second gives non-negative curvature of the frozen supporting functional; the touching argument then yields preservation of convexity. Thus the two results are distinct, but their signs come from the same covariance mechanism.

Our analysis covers a large family of problems and allows general prior distributions subject to mild conditions. In particular, we impose no conjugacy assumption. The main results are developed for discrete-time exponential-family observations by means of a continuous interpolation of the posterior family. When one member of the observation family is infinitely divisible, this interpolation is the actual posterior generated by an Esscher family of L\'evy processes. The convex-order comparison and preservation of convexity then hold in continuous time. This setting includes Brownian observations with an unknown drift, Poisson observations with an unknown intensity, and many finite- and infinite-activity jump models. In the Brownian case, the posterior probability process in \cite{EV} and the posterior mean process in \cite{EKV} have volatility coefficients that are non-increasing in time along fixed state levels.

The rest of the paper is organized as follows. Section~\ref{sec_cov} develops the posterior interpolation and the covariance inequalities underlying the main results. In Section~\ref{sec_main}, we prove the dynamic convex-order comparison and preservation of convexity and discuss their consequences. Section~\ref{sec_eg} presents the extension to L\'evy processes together with discrete- and continuous-time examples.

\section{Problem setup and the covariance inequality}\label{sec_cov}
We consider a one-dimensional exponential family for the observations $X_k$, $k\geq 1$. Let $\nu$ be a $\sigma$-finite measure on $\R$, and define
\[B(u):=\log\left(\int_{\R}\e^{ux}\nu(\ud x)\right),\quad N:=\left\{u\in\R:\int_{\R}\e^{ux}\nu(\ud x)<\infty\right\}.
\]
Thus $B(u)<\infty$ for $u\in N$. Write $S:=\operatorname{supp}\mu\subseteq N^\circ$. For $u\in N$, let
\[P_u(\ud x):=\e^{ux-B(u)}\nu(\ud x).\]
This is a probability distribution. We assume that, conditional on $\Theta=u$, each observation $X_k$ has law $P_u$.
Then $B'(u)=\E_u[X_1]$ and $B''(u)=\Var_u(X_1)>0$. For every bounded non-decreasing function $h$,
\[u_1<u_2\quad\implies\quad\E_{u_1}[h(X_1)]\leq\E_{u_2}[h(X_1)].\]
The same monotonicity holds for the sum of $k$ observations. Define $Y_n:=X_1+\cdots+X_n$. Conditional on $Y_n=y$, the posterior after $n$ observations is
\[\mu_{n,y}(\ud u)=\frac{\e^{uy-nB(u)}\mu(\ud u)}{\int_S \e^{uy-nB(u)}\mu(\ud u)}.\]
Thus, the pair $(n,Y_n)$ determines the entire posterior distribution. Let $T:S\to\R$ be non-decreasing, integrable under $\mu$, and not $\mu$-a.s.\ constant. For $t\geq0$, define the natural sufficient-statistic domain
\[\mathcal Y_t:=\left\{y\in\R:\int_S(1+|T(u)|)\e^{uy-tB(u)}\mu(\ud u)<\infty\right\}.\]
For $y\in\mathcal Y_t$, introduce the interpolated posterior
\begin{equation}\label{eq_t}
\mu_{t,y}(\ud u)=\frac{\e^{uy-tB(u)}\mu(\ud u)}{\int_S\e^{uy-tB(u)}\mu(\ud u)},\quad\Lambda_t(y):=\int_ST(u)\mu_{t,y}(\ud u).
\end{equation}
If $t$ is an integer, this is the usual posterior after $t$ observations. For a general exponential family, a non-integer value of $t$ gives only an interpolation, and there need not be an associated observation process. If, however, $P_{u_0}$ is infinitely divisible for some $u_0\in N^\circ$, then the interpolation is realized by an Esscher family of L\'evy processes; see Section~\ref{sec_continuous}.

By H\"older's inequality, the set $\mathcal Y_t$ is an interval. If $y_1<y_2$, the likelihood ratio of $\mu_{t,y_2}$ with respect to $\mu_{t,y_1}$ is strictly increasing in $u$; since $T$ is non-decreasing and not almost surely constant, $\Lambda_t(y_1)<\Lambda_t(y_2)$. Let $\mathcal Z_t:=\Lambda_t(\mathcal Y_t)$. Every $z\in\mathcal Z_t$ therefore determines a unique $y(t,z)\in\mathcal Y_t$. We write $\mu_{t,z}:=\mu_{t,y(t,z)}$. The inverse $z\mapsto y(t,z)$ is Borel on $\mathcal Z_t$. For the two state derivatives used below, we work on the smaller domain
\[\mathcal Y_t^{(2)}:=\operatorname{int}\left\{y\in\R:
\int_S(1+|T(u)|)(1+u^2)\e^{uy-tB(u)}\mu(\ud u)<\infty\right\},\quad\mathcal Z_t^{(2)}:=\Lambda_t(\mathcal Y_t^{(2)}).\]
On $\mathcal Z_t^{(2)}$, the inverse $y(t,\cdot)$ is continuous and the required moment bounds hold locally uniformly.
\begin{remark}\label{rem_markov}
The conditional distributions of future observations, and hence of future values of $\Pi^T$, depend on the past only through the pair $(n,\Pi_n^T)$. Thus $(\Pi_n^T)_{n\geq0}$ is a time-inhomogeneous Markov chain, with transition kernels defined in Section~\ref{sec_main}.
\end{remark}
Fix a time interval $[t_0,t_1]$ and a level $z$. Suppose that there are $y_-<y_+$ such that, for every $t\in[t_0,t_1]$, the equation $\Lambda_t(y_t)=z$ has a solution $y_t\in(y_-,y_+)$. At each of the four corner points $(t,y)\in\{t_0,t_1\}\times\{y_-,y_+\}$, assume that
\begin{equation}\label{assump_1}
\int_S(1+|T(u)|)(1+|u|+|B(u)|)\e^{uy-tB(u)}\mu(\ud u)<\infty.
\end{equation}
For each fixed $u$, the map $(t,y)\mapsto uy-tB(u)$ is affine, so its maximum over $[t_0,t_1]\times[y_-,y_+]$ is attained at one of the four corners. Condition~\eqref{assump_1} is used only to justify differentiation of the interpolated posterior family along the fixed-$T$ level curve.

We now carry out the calculation for which these assumptions are needed. Under \eqref{assump_1}, the map $(t,y)\mapsto\Lambda_t(y)$ is continuously differentiable, with partial derivatives
\begin{align}
\partial_y\Lambda_t(y)
&=\int_SuT(u)\mu_{t,y}(\ud u)
-\left(\int_Su\mu_{t,y}(\ud u)\right)
 \left(\int_ST(u)\mu_{t,y}(\ud u)\right)\notag\\
&=\Cov_{\mu_{t,y}}\left(T(\Theta),\Theta\right),\\
\partial_t\Lambda_t(y)
&=-\int_ST(u)B(u)\mu_{t,y}(\ud u)
+\left(\int_ST(u)\mu_{t,y}(\ud u)\right)
 \left(\int_SB(u)\mu_{t,y}(\ud u)\right)\notag\\
&=-\Cov_{\mu_{t,y}}(T(\Theta),B(\Theta)).
\end{align}
Since $T$ is non-decreasing, the first covariance above equals
\[\frac{1}{2}\E[(T(U)-T(V))(U-V)]>0,\]
where $U$ and $V$ are independent with distribution $\mu_{t,y}$. Strict positivity follows because $T$ is not $\mu$-a.s.\ constant and $\mu_{t,y}$ is equivalent to $\mu$. Hence $y\mapsto\Lambda_t(y)$ is strictly increasing, and the solution of the level equation, whenever it exists, is unique. The implicit function theorem now implies that $t\mapsto y_t$ is continuously differentiable. Differentiating $\Lambda_t(y_t)=z$ and using the two derivative identities above gives
\[
\frac{\ud y_t}{\ud t}
=\frac{\Cov_{\mu_{t,y_t}}(T(\Theta),B(\Theta))}
{\Cov_{\mu_{t,y_t}}(T(\Theta),\Theta)}.
\]
Write $\mu_t:=\mu_{t,y_t}$. Differentiating the posterior density in \eqref{eq_t} along the level curve gives
\begin{align*}
\frac{\ud}{\ud t}\mu_t(\ud u)
&=\frac{\ud}{\ud t}\left(\frac{\e^{uy_t-tB(u)}}{\int_S\e^{uy_t-tB(u)}\mu(\ud u)}\right)\mu(\ud u)\\
&=\left(u\frac{\ud y_t}{\ud t}-B(u)-\int_S\left(u\frac{\ud y_t}{\ud t}-B(u)\right)\mu_t(\ud u)\right)\mu_t(\ud u)\\
&=:G_t(u)\mu_t(\ud u).
\end{align*}
Therefore, for every test function $f$ satisfying $|f|\leq c(1+|T|)$, the map $t\mapsto\int_Sf \ud\mu_t$ is differentiable, with derivative $\int_SfG_t \ud\mu_t$. Since $B$ is convex, $G_t$ is concave. Moreover, $\mu_t$ has total mass one and $\int_ST(u)\mu_t(\ud u)=z$. Differentiating these two identities with respect to $t$ gives
\begin{equation}\label{eq_fix}
\int_SG_t(u)\mu_t(\ud u)=0,\quad
\int_ST(u)G_t(u)\mu_t(\ud u)=0.
\end{equation}
Hence $\Cov_{\mu_t}(T,G_t)=0$.
We first prove a monotonicity result used in the covariance calculation.
\begin{lemma}\label{lem_ratio}
Let $\mu$ be a probability measure supported on $S$, let $T:S\to\R$ be non-decreasing with $\int_S|T| \ud\mu<\infty$, and let $q:S\to[0,\infty)$ be bounded and non-decreasing. For $w\in\R$, define
\begin{align*}
P(w)&:=\iint_{u<w<v}(T(v)-T(u)) \mu(\ud u)\mu(\ud v),\\
R(w)&:=\iint_{u<w<v}q(u)q(v)(T(v)-T(u)) \mu(\ud u)\mu(\ud v).
\end{align*}
Then $R/P$ is non-decreasing on the set where $P(w)>0$. If $q$ is non-increasing, then $R/P$ is non-increasing on this set.
\end{lemma}

\begin{proof}
First, $P(w)\geq0$ for every $w$. Fix $w_1<w_2$ such that $P(w_1)P(w_2)>0$, and define $K(w,u,v):=\ind_{\{u<w<v\}}(T(v)-T(u))\geq0$, where $w\in\{w_1,w_2\}$ and $u,v\in S$. We first verify the MTP$_2$ condition
\begin{equation}\label{eq_tp}
K(x)K(x')\leq K(x\wedge x')K(x\vee x'),
\end{equation}
for $x=(w,u,v)$ and $x'=(w',u',v')$, where the minimum and maximum are taken coordinatewise. If $K(x)K(x')=0$, there is nothing to prove. Suppose that $K(x)K(x')>0$. Then $u<w<v$ and $u'<w'<v'$, with $T(u)<T(v)$ and $T(u')<T(v')$. Without loss of generality, assume that $u\leq u'$. If $v\leq v'$, then $x\wedge x'=(w\wedge w',u,v)$ and $x\vee x'=(w\vee w',u',v')$. The inequalities $u\leq u'<w'$ and $u<w$ give $u<w\wedge w'$, while $w<v\leq v'$ and $w'<v'$ give $w\vee w'<v'$. Thus both indicators equal one, and equality holds in \eqref{eq_tp}. If $v'\leq v$, the new pairs are $(u,v')$ and $(u',v)$. The inequalities $u\leq u'<w'$ and $u<w$ give $u<w\wedge w'$, while $w'<v'$ gives $w\wedge w'<v'$. Similarly, $u'<w'\leq w\vee w'$, and the inequalities $w<v$ and $w'<v'\leq v$ give $w\vee w'<v$. Thus both indicators again equal one. Moreover,
\[T(u)\leq T(u')<T(v')\leq T(v).\]
Consequently,
\begin{align*}
   & (T(v')-T(u))(T(v)-T(u'))-(T(v)-T(u))(T(v')-T(u'))\\
   =&(T(u')-T(u))(T(v)-T(v'))\geq 0.
\end{align*}
Thus \eqref{eq_tp} holds in all cases. Since $\int_S|T| \ud\mu<\infty$, we have $P(w_i)\leq2\int_S|T| \ud\mu<\infty$. Define a probability measure for $(W,U,V)$ by
\[\P(W=w_i,\ U\in\ud u,\ V\in\ud v)=\frac{K(w_i,u,v)}{P(w_1)+P(w_2)}\mu(\ud u)\mu(\ud v),\quad i=1,2.\]
Then $\P(W=w_i)=\frac{P(w_i)}{P(w_1)+P(w_2)}$ and $\E[q(U)q(V)\vert W=w_i]=\frac{R(w_i)}{P(w_i)}$. By \eqref{eq_tp}, the joint distribution of $(W,U,V)$ is TP$_2$. Since the functions $\ind_{\{W=w_2\}}$ and $q(U)q(V)$ are both non-decreasing, the association inequality for TP$_2$ measures \cite{FKG,KR} gives
\begin{align*}
    0&\leq \Cov\left(\ind_{\{W=w_2\}},q(U)q(V)\right)\\
    & = \E[\ind_{\{W=w_2\}}q(U)q(V) ]-\E[\ind_{\{W=w_2\}}]\E[q(U)q(V) ]\\
    & = \P(W=w_2)\frac{R(w_2)}{P(w_2)}- \P(W=w_2)\left(\P(W=w_1)\frac{R(w_1)}{P(w_1)}+\P(W=w_2)\frac{R(w_2)}{P(w_2)} \right)\\
    & = \P(W=w_1)\P(W=w_2)(\frac{R(w_2)}{P(w_2)}-\frac{R(w_1)}{P(w_1)}).
\end{align*}
Since $\P(W=w_1)\P(W=w_2)>0$, it follows that $\frac{R(w_2)}{P(w_2)}\geq\frac{R(w_1)}{P(w_1)}$. If $q$ is non-increasing, then $-q(U)q(V)$ is non-decreasing. Applying the same argument to $\ind_{\{W=w_2\}}$ and $-q(U)q(V)$ gives the reversed inequality.
\end{proof}
Observe that $R(w)=0$ whenever $P(w)=0$. Thus it suffices to consider $w$ for which $P(w)>0$, since the remaining values do not contribute to the covariance calculation below. We now state the covariance inequality that underlies both main results: after reweighting by a non-decreasing function, $T(\Theta)$ and a concave function $G(\Theta)$ are negatively correlated whenever they are orthogonal under the original measure $\mu$.
\begin{proposition}\label{prop_cov}
Let $\mu$ be a probability measure supported on $S$. Let $G:N^\circ\to\R$ be concave and let $T:S\to\R$ be non-decreasing. Suppose that
\[\int_S(|T|+|G|+|TG|) \ud\mu<\infty,\quad\Cov_\mu(T,G)=0.\]
Let $q:S\to[0,\infty)$ be bounded and non-decreasing, and suppose that $\int_Sq(u)\mu(\ud u)>0$. Then
\[\Cov_{q\mu/\int_S q \ud\mu}(T,G)\leq 0.\]
The inequality is reversed if $q$ is non-increasing.
\end{proposition}

\begin{proof}
Let $h=D^+G$ be the right derivative of the concave function $G$. Then $h$ is non-increasing and
\[G(v)-G(u)=\int_u^v h(w)\ud w\]
for $u,v\in S$ with $u<v$. Since $h$ is non-increasing, it changes sign at most once. If $G$ attains an interior maximum at $w^*\in[u,v]$, then $\int_u^v|h| \ud w=2G(w^*)-G(u)-G(v)$. If $G$ is monotone on $[u,v]$, then $\int_u^v|h| \ud w=|G(v)-G(u)|$. Thus there exists a finite constant $C=2\max\{G(w^*),0\}$ such that
\[\int_u^v|h(w)| \ud w\leq C+|G(u)|+|G(v)|.\]
Therefore, 
\begin{align*}
&\iint_{u<v}(T(v)-T(u))\int_u^v|h(w)|\ud w \mu(\ud u)\mu(\ud v)\\
\leq& 2C\int_S|T| \ud\mu +2\int_S|TG| \ud\mu +2\left(\int_S|T| \ud\mu\right) \left(\int_S|G| \ud\mu\right)<\infty.
\end{align*}
The bound remains finite after multiplication by the bounded factor $q(u)q(v)$. For independent random variables $U,V$ with distribution $\mu$, we have
\begin{align*}
\Cov_\mu(T,G)
&=\frac12\iint(T(v)-T(u))(G(v)-G(u)) \mu(\ud u)\mu(\ud v)\\
&=\iint_{u<v}(T(v)-T(u))(G(v)-G(u)) \mu(\ud u)\mu(\ud v).
\end{align*}
Using the definitions of $P$ and $R$ in Lemma~\ref{lem_ratio}, the identity $G(v)-G(u)=\int_u^v h(w)\ud w$, and Fubini's theorem, we obtain
\[\Cov_\mu(T,G) = \iint_{u<v}(T(v)-T(u))\left(\int_u^v h(w)\ud w\right)\mu(\ud u)\mu(\ud v)=\int_{N^\circ} h(w)P(w) \ud w.\]
Similarly,
\begin{align*}
\Cov_{q\mu/\int_S q\ud\mu}(T,G)&=\frac{1}{\left(\int_S  q\ud\mu\right)^2}\iint_{u<v}q(u)q(v)(T(v)-T(u))(G(v)-G(u))
\mu(\ud u)\mu(\ud v)\\
&=\frac{1}{\left(\int_Sq\ud\mu\right)^2}\int_{N^\circ}h(w)R(w) \ud w.
\end{align*}
Since $\Cov_\mu(T,G)=0$, we have $\int_{N^\circ}h(w)P(w) \ud w=0$. It suffices to consider $w$ for which $P(w)>0$, since only $P$ and $R$ appear in the integrals. By Lemma~\ref{lem_ratio}, $R/P$ is non-decreasing, whereas $h$ is non-increasing. Thus $\{h>0\}$ lies to the left of $\{h<0\}$. If either $\{P>0\}\cap\{h>0\}$ or $\{P>0\}\cap\{h<0\}$ is empty, then $hP$ has a fixed sign, and $\int_{N^\circ}h(w)P(w) \ud w=0$ implies that $hP=0$ almost everywhere. Since $0\leq R(w)\leq\|q\|_\infty^2P(w)$, it follows that $hR=0$ almost everywhere, and hence the covariance is zero. Otherwise, choose $\gamma$ such that $R(w)/P(w)\leq\gamma$ on $\{P>0\}\cap\{h>0\}$ and $R(w)/P(w)\geq\gamma$ on $\{P>0\}\cap\{h<0\}$. Then $h(w)(R(w)/P(w)-\gamma)\leq0$ everywhere. Consequently,
\[\int_{N^\circ} h(w)R(w)\ud w =\int_{N^\circ} h(w)(\frac{R}{P}(w) -\gamma)P(w)\ud w+\gamma\int_{N^\circ} h(w)P(w)\ud w\leq 0.
\]
Thus $\Cov_{q\mu/\int_Sq \ud\mu}(T,G)\leq0$. If $q$ is non-increasing, Lemma~\ref{lem_ratio} makes $R/P$ non-increasing, so the final inequality is reversed.
\end{proof}

\begin{lemma}\label{lem_frozen}
Let $\mu$ be a probability measure on $S$. Let $T:S\to\R$ be non-decreasing, and let $\psi:N^\circ\to\R$ satisfy
\[\int_S (|T|+|\psi|+|T\psi|)\ud\mu<\infty,\quad \int_S\psi\ud\mu=\int_ST\psi\ud\mu=0.\]
Let $q:S\to[0,1]$ be non-decreasing, and suppose that $Q:=\int_Sq(u)\mu(\ud u)\in(0,1)$. Define
\[z^+:=\frac{\int_ST(u)q(u)\mu(\ud u)}{Q},\quad
z^-:=\frac{\int_ST(u)(1-q(u))\mu(\ud u)}{1-Q}.\]
Then $z^-\leq \int_ST(u)\mu(\ud u)\leq z^+$. If $\psi$ is concave, then for every $a\in[z^-,z^+]$,
\[\int_S(T(u)-a)q(u)\psi(u)\mu(\ud u)\leq 0.\]
If $\psi$ is convex, the inequality is reversed.
\end{lemma}
\begin{proof}
Define $z:=\int_ST(u)\mu(\ud u)$. Since $T$ and $q$ are both non-decreasing, 
\[0\leq \Cov_\mu(T,q)=
\int_STq\ud\mu-\left(\int_ST\ud\mu\right)\left(\int_Sq\ud\mu\right)=Q(z^+-z).\]
Hence $z\leq z^+$. 
Moreover, since 
\[z = \int_Sq(u)\mu(\ud u)z^++(1-\int_Sq(u)\mu(\ud u))z^-,\]
this implies that $z^-\leq z$. Suppose that $\psi$ is concave. At $a=z^+$,
\[\int_S(T(u)-z^+)q(u)\psi(u)\mu(\ud u) = Q\Cov_{q\mu/\int_Sq\ud\mu}(T,\psi)\leq 0\]
by Proposition~\ref{prop_cov}. At $a=z^-$, the identity $\int_S(T(u)-z^-)\psi(u)\mu(\ud u)=0$ gives
\begin{align*}
    \int_S(T(u)-z^-)q(u)\psi(u)\mu(\ud u)&=-\int_S(T(u)-z^-)(1-q(u))\psi(u)\mu(\ud u)\\
    & = -\left(1-Q\right)\Cov_{(1-q)\mu/(1-\int_Sq\ud\mu)}(T,\psi)\leq 0,
\end{align*}
where the inequality follows because $1-q$ is non-increasing and Proposition~\ref{prop_cov} gives
\[\Cov_{(1-q)\mu/(1-\int_Sq \ud\mu)}(T,\psi)\geq 0.
\]
The map $a\mapsto\int_S(T(u)-a)q(u)\psi(u)\mu(\ud u)$ is affine and non-positive at both endpoints $z^-$ and $z^+$. Therefore, $\int_S(T(u)-a)q(u)\psi(u)\mu(\ud u)\leq0$ for every $a\in[z^-,z^+]$. The argument for convex $\psi$ is analogous.
\end{proof}

\section{Convex order comparison and preservation of convexity}\label{sec_main}
We now apply Proposition~\ref{prop_cov} to the continuous-time posterior interpolation. Recall that, for probability distributions $\alpha$ and $\beta$ with the same mean, $\alpha\leq_{\cx}\beta$ is equivalent to
\begin{equation}\label{eq_call_cx}
\int(x-a)^+\alpha(\ud x)\leq\int(x-a)^+\beta(\ud x)\quad\text{for every }a\in\R.
\end{equation}
See \cite{SS}. Equivalently, by Strassen's theorem, there exist random variables $X\sim\alpha$ and $Y\sim\beta$ on a common probability space such that $\E[Y\vert X]=X$. Thus one may write $Y=X+Z$, where $\E[Z\vert X]=0$, and $\beta$ is a mean-preserving spread of $\alpha$; see \cite{S}.

We first define the posterior transition kernel without conditioning on a possibly null event. Fix $t\geq0$, an integer $k\geq1$, and $u_*\in N^\circ$. Let $P_u^{(k)}$ be the law, conditional on $\Theta=u$, of the sufficient statistic from the next $k$ observations. Then
\[\ell_k(u,x):=\frac{\ud P_u^{(k)}}{\ud P_{u_*}^{(k)}}(x)
=\exp\!\left\{(u-u_*)x-k\bigl(B(u)-B(u_*)\bigr)\right\}.\]
For $z\in\mathcal Z_t$, set
\[\rho_{t,k}(z,x):=\int_S\ell_k(u,x)\mu_{t,z}(\ud u),
\quad
\eta_{t,k}(z,x):=\int_S|T(u)|\ell_k(u,x)\mu_{t,z}(\ud u),
\]
and define the updated posterior mean by
\begin{equation}\label{eq_discrete_update}
M_{t,z}^{(k)}(x)
:=
\begin{cases}
\displaystyle
\frac{\int_ST(u)\ell_k(u,x)\mu_{t,z}(\ud u)}
{\rho_{t,k}(z,x)},
&0<\rho_{t,k}(z,x)<\infty,\ \eta_{t,k}(z,x)<\infty,\\[10pt]
z,&\text{otherwise}.
\end{cases}
\end{equation}
On the full-measure set where the ratio in \eqref{eq_discrete_update} is defined, the posterior formula and the likelihood ratio above give
\[M_{t,z}^{(k)}(x)=\Lambda_{t+k}\bigl(y(t,z)+x\bigr).\]
The predictive law of the future sufficient statistic is $\int_SP_u^{(k)}(\ud x)\mu_{t,z}(\ud u)$. For $C\in\mathcal B(\R)$, define
\begin{equation}\label{eq_discrete_kernel}
K_{t,k}^T(z,C):=\int_{\R}\ind_C\!\left(M_{t,z}^{(k)}(x)\right)
\int_SP_u^{(k)}(\ud x)\mu_{t,z}(\ud u),\quad z\in\mathcal Z_t,
\end{equation}
and let
\[\mathcal K_{t,k}^Tf(z):=\int_{\R}f(w)K_{t,k}^T(z,\ud w).\]
For each $z\in\mathcal Z_t$, Tonelli's theorem shows that the first case in \eqref{eq_discrete_update} holds on a $P_{u_*}^{(k)}$-full set and that the convention on its complement does not affect the kernel. The positive and negative parts of the signed numerator in \eqref{eq_discrete_update} are parameter integrals of nonnegative Borel functions; taking their difference on the Borel set $\{\eta_{t,k}<\infty\}$ shows that $M_{t,z}^{(k)}(x)$ is jointly Borel in $(z,x)$. Thus \eqref{eq_discrete_kernel} is a Borel transition kernel on $\mathcal Z_t$. It also satisfies
\begin{equation}\label{eq_kernel_mean}
\int_{\R}wK_{t,k}^T(z,\ud w)=z.
\end{equation}
We will repeatedly use the following consequence of Bayes' formula: for every $a\in\R$ and $A\in\mathcal B(\R)$,
\begin{equation}\label{eq_frozen_event}
\int_{\R}\bigl(M_{t,z}^{(k)}(x)-a\bigr)\ind_A(x)\int_SP_u^{(k)}(\ud x)\mu_{t,z}(\ud u)
=\int_S(T(u)-a)P_u^{(k)}(A)\mu_{t,z}(\ud u).
\end{equation}
At integer times, Bayes' formula and Tonelli's theorem give $Y_n\in\mathcal Y_n$, and hence $\Pi_n^T\in\mathcal Z_n$, almost surely. Moreover, for every $C\in\mathcal B(\R)$,
\[\P\bigl(\Pi_{n+k}^T\in C\vert\mathcal F_n\bigr)
=K_{n,k}^T(\Pi_n^T,C)
\quad\text{a.s.}
\]
We use $\E_{t,z}$ only as shorthand for expectation in the Bayesian experiment initialized from $\mu_{t,z}$, rather than as conditioning on $\{\Pi_t^T=z\}$.

\begin{theorem}\label{thm1}
Let $m,n,k$ be integers with $m<n$ and $k\geq1$, and let $z\in\mathcal Z_m\cap\mathcal Z_n$. Suppose that there are $y_-<y_+$ such that $y(r,z)\in(y_-,y_+)$ for every $r\in[m,n]$ and that \eqref{assump_1} holds at the four corners of $[m,n]\times[y_-,y_+]$. Then
\[
K_{m,k}^T(z,\cdot)\geq_{\cx}K_{n,k}^T(z,\cdot).
\]
\end{theorem}
\begin{proof}
Fix $m\leq t<s\leq n$ and $a\in\R$, and write $c_a(w):=(w-a)^+$. For $t\leq r\leq s$, let $\mu_r:=\mu_{r,z}$ and $M_r(x):=M_{r,z}^{(k)}(x)$. The map $M_s$ is non-decreasing on a $P_{u_*}^{(k)}$-full set. Hence its superlevel set agrees $P_{u_*}^{(k)}$-almost everywhere with an upper Borel set. All the laws $P_u^{(k)}$ are equivalent to $P_{u_*}^{(k)}$, so this replacement changes none of the probabilities below. Choose such an upper set $A$ for $\{x:M_s(x)>a\}$, and let $q(u):=P_u^{(k)}(A)$. Then $q$ takes values in $[0,1]$ and is non-decreasing. Keep $A$ fixed as $r$ varies from $t$ to $s$, and define $H(r):=\int_S(T(u)-a)q(u)\mu_r(\ud u)$. By \eqref{eq_frozen_event}, this is the frozen-event integral on the left-hand side of that identity.
By the choice of $A$, $H(s)=\mathcal K_{s,k}^Tc_a(z)$. At time $t$, the fixed set $A$ need not maximize the corresponding event integral, and hence $H(t)\leq\mathcal K_{t,k}^Tc_a(z)$. It therefore suffices to show that $H$ is non-increasing. We have
\[H'(r)=\int_S(T(u)-a)q(u)G_r(u)\mu_r(\ud u).\]
To determine its sign, set $Q(r):=\int_Sq(u)\mu_r(\ud u)$. If $q=0$ $\mu$-a.e., then $H\equiv0$; if $q=1$ $\mu$-a.e., then $H\equiv z-a$. Otherwise, when $0<Q(r)<1$, define $z^+(r)$ and $z^-(r)$ as above, with $(\mu,Q)$ replaced by $(\mu_r,Q(r))$.
Since $\int_ST(u)\mu_r(\ud u)=z$, we can write $z$ and $H$ as
\[z=Q(r)z^+(r)+(1-Q(r))z^-(r),\quad H(r)=Q(r)(z^+(r)-a).\]
We first compute the derivatives of $z^+$ and $z^-$. Since
\[\frac{\ud}{\ud r}\int_S T(u)q(u)\mu_r(\ud u)=\int_S T(u)q(u)G_r(u)\mu_r(\ud u)\]
and $Q'(r)=\int_Sq(u)G_r(u)\mu_r(\ud u)$, we obtain $(z^+)'(r)=\Cov_{q\mu_r/Q(r)}(T,G_r)\leq0$. The inequality follows from Proposition~\ref{prop_cov}, because $G_r$ is concave, $\Cov_{\mu_r}(T,G_r)=0$ by \eqref{eq_fix}, and $q$ is non-decreasing. Similarly, since $1-q$ is non-increasing,
\[(z^-)'(r)=\Cov_{(1-q)\mu_r/(1-Q(r))}(T,G_r)\geq 0.\]
At time $s$, the maximizing property of $A$ and \eqref{eq_frozen_event} give $z^-(s)\leq a\leq z^+(s)$.
Since $z^+$ is non-increasing and $z^-$ is non-decreasing, we have $z^-(r)\leq z^-(s)\leq a\leq z^+(s)\leq z^+(r)$ for $t\leq r\leq s$. Hence $a\in[z^-(r),z^+(r)]$ for every $r\in[t,s]$.
The function $G_r$ is concave and satisfies the orthogonality relations \eqref{eq_fix}. Applying Lemma~\ref{lem_frozen} to the derivative above, with $\mu=\mu_r$ and $\psi=G_r$, gives $H'(r)\leq0$. It follows that
\[
\mathcal K_{s,k}^Tc_a(z)=H(s)\leq H(t)\leq\mathcal K_{t,k}^Tc_a(z).
\]
For each $r\in\{t,s\}$, the $k$-step transition law is integrable and has mean $z$ by \eqref{eq_kernel_mean}. Since $a$ is arbitrary, the convex-order comparison follows from \eqref{eq_call_cx}. Taking $t=m$ and $s=n$ in the interpolation proves the result.
\end{proof}
As a consequence, suppose that the assumptions of Theorem~\ref{thm1} hold. Let $v$ be convex and suppose that $\mathcal K_{r,k}^T|v|(z)<\infty$ for $r=m,n$. Then, for $m<n$,
\[\mathcal K_{m,k}^Tv(z)\geq\mathcal K_{n,k}^Tv(z).
\]
From a decision-theoretic perspective, suppose that, after observing the next $k$ outcomes, a decision maker chooses an action $\alpha\in\mathcal A$ and receives the payoff $g(\alpha,u)=a(\alpha)+b(\alpha)T(u)$. The optimized payoff then depends on the posterior only through $z$, and
\[v(z):=\sup_{\alpha\in\mathcal A}\{a(\alpha)+b(\alpha)z\}
\]
is convex. The value of this information at time $t$ is $\mathcal K_{t,k}^Tv(z)-v(z)$. If the $k$ observations can be purchased at a fixed cost $c_k\geq0$, the decision maker purchases them precisely when $\mathcal K_{t,k}^Tv(z)-v(z)\geq c_k$. Thus, at a fixed posterior state $z$, the decision maker's willingness to pay for information decreases with information time.
\begin{remark}\label{rem_information}
The convex-order comparison is not fundamentally tied to the number of i.i.d.\ observations. Suppose that the current posteriors have the form $\mu_{t,y}(\ud u)\propto\exp\{uy-C_t(u)\}\mu(\ud u)$, where $C_s-C_t$ is convex whenever $t<s$. Along a $\Pi^T$-level curve, the log-likelihood ratio is then concave in $u$, which is exactly the property used in the proof of Theorem~\ref{thm1}. If both current posteriors are exposed to the same future experiment $Y$, and if, for every $a\in\R$, the function $u\mapsto\P_u\left(\E[T(\Theta)\vert Y]>a\right)$ is non-decreasing, the same convex-order comparison follows. Thus the observations may be conditionally independent without being identically distributed.

For example, suppose that $X_i\vert\Theta=u\sim\operatorname{Poisson}(a_i\e^u)$, where the parameters $a_i>0$ may differ. Then the posterior after $n$ observations is proportional to
\[\exp\{uY_n-A(n)\e^u\}\mu(\ud u),\]
where $A(n)=\sum_{i=1}^na_i$ and $Y_n=\sum_{i=1}^nX_i$. Thus heterogeneous observations induce a clock change from sample time $n$ to information time $A(n)$. If the two current posteriors are exposed to the same future Poisson experiment, the convex-order comparison continues to hold. The comparison is therefore naturally indexed by accumulated information rather than by the number of observations. The i.i.d.\ exponential-family model corresponds to the special case $a_i=a$.
\end{remark}
Theorem~\ref{thm1} establishes a family of distributions that decreases in convex order along a posterior level curve. Read in reverse information time, this family can be realized as the marginals of a martingale by Kellerer's theorem~\cite{K}. The following corollary gives a local measure of this contraction.
\begin{corollary}\label{cor_cov_monot}
Under the assumptions of Theorem~\ref{thm1}, along a fixed level $\Lambda_t(y_t)=z$, the map $t\mapsto\Cov_{\mu_{t,y_t}}(T(\Theta),\Theta)$ is non-increasing.
\end{corollary}
\begin{proof}
Write $\mu_t:=\mu_{t,y_t}$. For $c\in\R$, define
\[F_c(t)
:=\Cov_{\mu_t}(T(\Theta),\ind_{\{\Theta>c\}})=\int_S (T(u)-z)\ind_{\{u>c\}}\mu_t(\ud u).\]
Since $T$ and $u\mapsto\ind_{\{u>c\}}$ are non-decreasing, we have $F_c(t)\geq0$. If $0<\mu_t((c,\infty))<1$, then
\[F_c'(t)=\int_S(T(u)-z)\ind_{\{u>c\}}G_t(u)\mu_t(\ud u).\]
By \eqref{eq_fix}, the concave function $G_t$ satisfies the orthogonality conditions in Lemma~\ref{lem_frozen}. In that lemma, take $q(u)=\ind_{\{u>c\}}$, $a=z$, and $\psi=G_t$. This gives $F_c'(t)\leq0$. If $\mu_t((c,\infty))\in\{0,1\}$, then $F_c(t)=0$, and the same monotonicity conclusion holds. Hence $F_c(t)\geq F_c(s)$ whenever $t\leq s$. Therefore,
\begin{align*}
\Cov_{\mu_t}(T(\Theta),\Theta)&=\frac{1}{2}\iint(T(v)-T(u))(v-u)\mu_t(\ud u)\mu_t(\ud v)\\
&=\int_{\R}\Cov_{\mu_t}(T(\Theta),\ind_{\{\Theta>c\}})\ud c\\
&=\int_{\R}F_c(t)\ud c\geq \int_{\R}F_c(s)\ud c = \Cov_{\mu_s}(T(\Theta),\Theta).
\end{align*}
\end{proof}
An equivalent geometric interpretation is that the level curves spread apart as information accumulates. Observe that $\partial_z y(t,z)=1/\Cov_{\mu_{t,y(t,z)}}(T(\Theta),\Theta)$. Corollary~\ref{cor_cov_monot} then implies that $\partial_z y(t,z)$ is non-decreasing in $t$. Consequently, for $z_1<z_2$, the difference $y(t,z_2)-y(t,z_1)=\int_{z_1}^{z_2}\partial_z y(t,z) \ud z$ is non-decreasing in $t$. Thus distinct $\Pi^T$-level curves spread apart, recovering \cite[Corollary~4.2]{EW}. For $T(u)=\ind_{\{u>\theta_0\}}$, write $\pi=z$. The intermediate quantity $F_c(t)$ is then given by
\begin{align*}
F_c(t)=
\begin{cases}
\pi\P_{t,\pi}(\Theta\leq c), & c<\theta_0,\\
(1-\pi)\P_{t,\pi}(\Theta>c), & c\geq\theta_0.
\end{cases}
\end{align*}
The fact that $F_c$ decreases in $t$ precisely recovers \cite[Theorem~4.1]{EW}. For $T(u)=u$, the same result says that posterior-mean level curves spread apart; equivalently, the posterior variance decreases along each level curve.

Theorem~\ref{thm1} compares posterior kernels at different times. Many stopping problems involve either concave cost functionals or convex gain functionals, and the convex order of the $k$-step transition has important implications for the time monotonicity of the value function and the structure of the stopping boundaries. To use this comparison recursively, however, the transition operator must also preserve concavity or convexity. We now establish the corresponding convexity-preservation property.

Fix $t\geq0$. On $\mathcal Y_t^{(2)}$, the map $\Lambda_t$ is twice continuously differentiable. Moreover, because $T$ is non-decreasing and not $\mu$-a.s.\ constant, $\partial_y\Lambda_t>0$. We use the inverse relation $y=y(t,z)$ and the state-indexed posterior $\mu_{t,z}$ defined above throughout the following calculation.
\begin{lemma}\label{lem_2nd_deriv}
Along the parametrization $y=y(t,z)$, the first and second derivatives of $\mu_{t,y}$ satisfy
\begin{equation}\label{eq_state_derivatives}
\frac{\partial}{\partial z}\mu_{t,y}(\ud u)
=\frac{u-\E_{\mu_{t,y}}[\Theta]}{\Cov_{\mu_{t,y}}(T(\Theta),\Theta)}\mu_{t,y}(\ud u),
\quad
\frac{\partial^2}{\partial z^2}\mu_{t,y}(\ud u)=Q_z(u)\mu_{t,y}(\ud u).
\end{equation}
where
\[Q_z(u)=\frac{ (u-\E_{\mu_{t,y}}[\Theta] )^2-\Var_{\mu_{t,y}}(\Theta)}{\Cov^2_{\mu_{t,y}}(T(\Theta),\Theta)}-
\frac{\E_{\mu_{t,y}} \left[(T(\Theta)-z) (\Theta-\E_{\mu_{t,y}}[\Theta] )^2\right]}{\Cov^3_{\mu_{t,y}}(T(\Theta),\Theta)}(u-\E_{\mu_{t,y}}[\Theta] ).\]
The function $Q_z$ is convex and satisfies
\begin{equation}\label{eq_Q}
\int_S Q_z(u) \mu_{t,y}(\ud u)=0,\quad\int_S T(u)Q_z(u) \mu_{t,y}(\ud u)=0.
\end{equation}
\end{lemma}

\begin{proof}
The map $z\mapsto y$ is twice continuously differentiable. Since $\partial_y \Lambda_t(y) = \Cov_{\mu_{t,y}}(T(\Theta),\Theta)$, its first derivative is $1/\Cov_{\mu_{t,y}}(T(\Theta),\Theta)$; differentiating \eqref{eq_t} therefore gives the first identity in \eqref{eq_state_derivatives}. Moreover,
\[\frac{\partial}{\partial z}\E_{\mu_{t,y}}[\Theta]=\int_S u 
\frac{u-\E_{\mu_{t,y}}[\Theta]}{\Cov_{\mu_{t,y}}(T(\Theta),\Theta)} \mu_{t,y}(\ud u)=\frac{\Var_{\mu_{t,y}}(\Theta)}{\Cov_{\mu_{t,y}}(T(\Theta),\Theta)},\]
and
\begin{align*}
\frac{\partial}{\partial z}\Cov_{\mu_{t,y}}(T(\Theta),\Theta)&=
\frac{\E_{\mu_{t,y}}\left[(T(\Theta)-z)(\Theta-\E_{\mu_{t,y}}[\Theta])^2\right]}{\Cov_{\mu_{t,y}}(T(\Theta),\Theta)}\\
&\quad-\E_{\mu_{t,y}}\left[\Theta-\E_{\mu_{t,y}}[\Theta]
\right]-\frac{\ud}{\ud z}\E_{\mu_{t,y}}[\Theta] \E_{\mu_{t,y}}[T(\Theta)-z]\\
&=\frac{\E_{\mu_{t,y}}\left[(T(\Theta)-z)(\Theta-\E_{\mu_{t,y}}[\Theta])^2\right]}{\Cov_{\mu_{t,y}}(T(\Theta),\Theta)}
\end{align*}
since we are on the $z$-level curve. A second differentiation gives the second identity in \eqref{eq_state_derivatives} and the stated formula for $Q_z$. This covariance is positive, so $Q_z$ is convex because its leading quadratic coefficient is positive. Finally, differentiating twice the normalization in \eqref{eq_t} and the identity $\int_ST(u)\mu_{t,z}(\ud u)=z$ gives \eqref{eq_Q}.
\end{proof}
\begin{theorem}\label{thm_convex}
Fix $t\geq0$ and an integer $k\geq1$, and let $J$ be a nonempty open interval contained in $\mathcal Z_t^{(2)}$. If $f:\R\to\R$ is finite and convex and $\mathcal K_{t,k}^T|f|(z)<\infty$ for $z\in J$, then the map $z\mapsto\mathcal K_{t,k}^Tf(z)$ is convex on $J$.
\end{theorem}
\begin{proof}
Fix $a\in\R$ and write $c_a(w):=(w-a)^+$. Set $V_a(z):=\mathcal K_{t,k}^Tc_a(z)$. For every $A\in\mathcal B(\R)$, define
\[q_A(u):=P_u^{(k)}(A),\quad H_A(z):=\int_S(T(u)-a)q_A(u)\mu_{t,z}(\ud u).\]
By \eqref{eq_frozen_event}, $H_A$ is also the event integral on the left-hand side of that identity. Consequently,
\[V_a(z)=\sup_{A\in\mathcal B(\R)}H_A(z),\]
and the supremum is attained, up to predictive null sets, by
$\{x:M_{t,z}^{(k)}(x)>a\}$. We first verify the continuity needed in the supporting-function argument. Let $J_0$ be a compact interval contained in $J$, and write
\[\bar u(z):=\int_Su\mu_{t,z}(\ud u),
\quad\gamma(z):=\Cov_{\mu_{t,z}}(T(\Theta),\Theta)>0.
\]
For every Borel set $A$, the moment bounds defining $\mathcal Y_t^{(2)}$ justify differentiation under the integral, and Lemma~\ref{lem_2nd_deriv} gives
\[H_A'(z)=\frac{1}{\gamma(z)}
\int_S(T(u)-a)q_A(u)(u-\bar u(z))\mu_{t,z}(\ud u).
\]
Since $0\leq q_A\leq1$,
\begin{equation}\label{eq_uniform_lipschitz}
|H_A'(z)|\leq\frac{\int_S|T(u)-a| |u-\bar u(z)|\mu_{t,z}(\ud u)}{\gamma(z)}.
\end{equation}
The right-hand side is independent of $A$ and is bounded on $J_0$: the regular domain gives locally uniform moment bounds, while the continuous, strictly positive function $\gamma$ is bounded away from zero on $J_0$. Denote this bound by $L_{J_0,a}$. Since $V_a(z)=\sup_{A\in\mathcal B(\R)}H_A(z)$, for $z,z'\in J_0$,
\[|V_a(z)-V_a(z')|\leq\sup_{A\in\mathcal B(\R)}|H_A(z)-H_A(z')|\leq L_{J_0,a}|z-z'|.\]
Thus $V_a$ is locally Lipschitz on $J$. Notice that the constant is uniform over the events in the envelope; no differentiation of the supremum is being used.

We next construct a supporting functional at an arbitrary $z_0\in J$. We have $M_{t,z_0}^{(k)}$ non-decreasing on the $P_{u_*}^{(k)}$-full set where \eqref{eq_discrete_update} has its ratio form. Choose an upper Borel set $A_0$ that agrees $P_{u_*}^{(k)}$-almost everywhere with $\{x:M_{t,z_0}^{(k)}(x)>a\}$, and put $q:=q_{A_0}$. The monotone-likelihood-ratio property implies that $q$ is non-decreasing. Moreover,
\[H_{A_0}(z)\leq V_a(z),\quad z\in J,\quad H_{A_0}(z_0)=V_a(z_0).\]
Taking the second derivative at $z_0$ gives
\[H_{A_0}''(z_0)=\int_S(T(u)-a)q(u)Q_{z_0}(u)\mu_{t,z_0}(\ud u).\]
Let $p_0:=\int_Sq(u)\mu_{t,z_0}(\ud u)$. If $p_0=0$, equivalence of the exponentially tilted posteriors implies that $q=0$ almost everywhere under every $\mu_{t,z}$, so $H_{A_0}\equiv0$. If $p_0=1$, the same equivalence gives $q=1$ almost everywhere and $H_{A_0}(z)=z-a$. Hence $H_{A_0}''(z_0)=0$ in either case.
Suppose now that $0<p_0<1$. By the choice of $A_0$ and \eqref{eq_frozen_event}, $a$ lies between the two conditional $T$-means as in Lemma \ref{lem_frozen}, with $(\mu,Q)=(\mu_{t,z_0},p_0)$. By \eqref{eq_Q}, the convex function $Q_{z_0}$ satisfies the orthogonality conditions in Lemma~\ref{lem_frozen}. Applying that lemma with $\psi=Q_{z_0}$ to the second derivative identity above gives $H_{A_0}''(z_0)\geq0$. We have therefore found, at every $z_0\in J$, a twice differentiable function $H_{A_0}$ that touches $V_a$ from below at $z_0$ and has non-negative second derivative there.

Suppose that $V_a$ is not convex. Then there are $z_1<z_2$ in $J$ and $z\in(z_1,z_2)$ such that $V_a(z)>l(z)$, where
\[
l(z):=\frac{z_2-z}{z_2-z_1}V_a(z_1)
+\frac{z-z_1}{z_2-z_1}V_a(z_2).
\]
Choose $\varepsilon>0$ sufficiently small that
\[z\mapsto V_a(z)-l(z)-\varepsilon(z-z_1)(z_2-z)\]
attains a positive maximum $m_0$ at some $\bar z\in(z_1,z_2)$, and define
$\varphi(z):=l(z)+m_0+\varepsilon(z-z_1)(z_2-z)$. Then $V_a\leq\varphi$ on $[z_1,z_2]$, with equality at $\bar z$, and $\varphi''(\bar z)=-2\varepsilon<0$. Let $H_{A_0}$ be the supporting functional constructed at $z_0=\bar z$. Since $H_{A_0}\leq V_a\leq\varphi$ and equality holds at $\bar z$, the function $\varphi-H_{A_0}$ has a local minimum there. Consequently, $\varphi''(\bar z)\geq H_{A_0}''(\bar z)\geq 0$, a contradiction. Hence $V_a$ is convex.

Finally, let $\ell$ be an affine supporting function of $f$. Then $f-\ell\geq0$ has the standard representation as a positive mixture of call and put functions; see Theorem~A.3.1 in \cite{SS}. Put functions differ from call functions by an affine function, and the posterior operator preserves affine functions by \eqref{eq_kernel_mean}. Tonelli's theorem and the assumption $\mathcal K_{t,k}^T|f|<\infty$ therefore extend the convexity conclusion from calls to $f$.
\end{proof}
Theorem~\ref{thm_convex} establishes convexity preservation directly from the Bayesian posterior structure. Related properties for parabolic and jump-diffusion operators have been studied extensively; see \cite{JT,ET}. In the Brownian model below, $\Pi^T$ is itself a one-dimensional diffusion, so there is a direct connection with this literature. The present proof also applies in discrete time without requiring generator conditions.

The extension discussed in Remark~\ref{rem_information} also applies to the convexity-preservation result. We retain the i.i.d.\ exponential-family formulation because it has a simple structure and is natural for the intended applications.

The two main theorems play different roles. Theorem~\ref{thm1} compares a fixed future block of observations for every convex terminal value. Theorem~\ref{thm_convex} shows that posterior updating preserves convexity, so continuation values generated by the Bellman equation remain convex and can be used at earlier times. Together, these results yield time-monotonicity properties for stopping problems. In a control problem, one must additionally check that the control does not alter the law of future observations and that optimization preserves the required convexity.

\section{Examples and the continuous-time case}\label{sec_eg}
\subsection{Examples in discrete time}\label{sec_examples}
We give three examples illustrating different choices of the posterior functional $T$. In each example, the prior is arbitrary subject to the standing assumptions.
\begin{example}\label{eg1}
Fix $\theta_0\in S$ and let $T(u)=\mathbf{1}_{(\theta_0,\infty)}(u)$. Then 
\[\Pi_n^T
=\E\left[\mathbf{1}_{\{\Theta>\theta_0\}}\vert\mathcal F_n\right]
=\P(\Theta>\theta_0\vert\mathcal F_n).\]
The process $\Pi_n^T$ is precisely the posterior probability process in the Bayesian testing problem of \cite{EW}. The function $T$ is bounded and non-decreasing, and it is nonconstant whenever $0<\mu((\theta_0,\infty))<1$. Thus, under the standing assumptions, this posterior probability is a special case of the posterior functionals considered here. In particular, Theorem~\ref{thm1} verifies Assumption~5.1 of \cite{EW} and thereby proves their Conjecture~6.1.
\end{example}

\begin{example}\label{eg2}
Let $T(u)=u$. Then $\Pi_n^T=\hat\Theta_n:=\E[\Theta\vert\mathcal F_n]$ is the posterior mean of the unknown parameter. Consider an investment problem in which the decision maker learns about the unknown profitability $\Theta$ and chooses both when to stop learning and whether to undertake a project, with discount rate $r>0$. If she stops at $\tau$ and chooses $D_\tau\in\{0,1\}$, where $D_\tau=1$ denotes investment, her payoff is $D_\tau(\Theta-K)$. Conditioning on $\mathcal F_\tau$ gives
\[\E[D_\tau(\Theta-K)\vert\mathcal F_\tau]=D_\tau(\hat\Theta_\tau-K).\]
The optimal investment decision at time $\tau$ is therefore $D_\tau^*=\ind_{\{\hat\Theta_\tau>K\}}$, and the resulting payoff is $(\hat\Theta_\tau-K)^+$. The investment problem can thus be written as the optimal stopping problem
\[V(n,z)=\sup_{\tau\geq n}\E_{n,z}\left[e^{-r(\tau-n)}(\hat\Theta_\tau-K)^+\right].\]
By Remark~\ref{rem_markov}, $(n,\hat\Theta_n)$ forms a Markov state. Let $g(x):=(x-K)^+$. First consider the problem with at most $j$ future observations. Set $V_0(n,x)=g(x)$ and define recursively
\[V_{j+1}(n,x)=\max\left\{g(x),e^{-r}\mathcal K_{n,1}^T V_j(n+1,\cdot)(x)\right\}.\]
Since $g$ is convex, Theorem~\ref{thm_convex} and induction show that $x\mapsto V_j(n,x)$ is convex for every $n$ and $j$. Moreover, for $m\leq n$, Theorem~\ref{thm1} gives
\[\mathcal K_{m,1}^T V_j(m+1,\cdot)(x)\geq\mathcal K_{n,1}^T V_j(m+1,\cdot)(x)\geq \mathcal K_{n,1}^T V_j(n+1,\cdot)(x),\]
and hence $V_{j+1}(m,x)\geq V_{j+1}(n,x)$. Under the usual integrability conditions, letting $j\to\infty$ yields $V(m,x)\geq V(n,x)$ for $m\leq n$. Thus, conditional on the same current posterior mean, the option to continue learning before making the investment decision is more valuable when fewer observations have been collected.
\end{example}
\begin{example}\label{eg3}
Let $T(u)=B'(u)$. Recall that $B'(u)=\E_u[X_1]$ and $B''(u)=\Var_u(X_1)\geq0$, so $T$ is non-decreasing. Moreover,
\[\Pi_n^T=\E[B'(\Theta)\vert\mathcal F_n]=\E[X_{n+1}\vert\mathcal F_n],\]
so $\Pi_n^T$ is the predicted mean of the next observation. For example, for Poisson observations the intensity parameter is $\lambda=\e^u$. Here $B(u)=\e^u$, $T(u)=B'(u)=\e^u$, and $\Pi_n^T=\E[\e^\Theta\vert\mathcal F_n]=\E[X_{n+1}\vert\mathcal F_n]$ is the predicted future count under the current posterior. Theorem~\ref{thm1} then implies that these forecasts stabilize as information accumulates, even when conditioning on the same current forecast.
\end{example}
Another useful class is obtained by taking $T(u):=\E_u[h(X_1)]$ for a bounded non-decreasing function $h$; then $\Pi_n^T=\E[h(X_{n+1})\vert\mathcal F_n]$. An example of this form appears in the continuous-time section; see Example~\ref{eg1_BM}. The next example shows that the monotonicity assumption on $T$ cannot be dropped from Theorem~\ref{thm1}, even with a three-point prior and Bernoulli observations. Without this assumption, the convex-order conclusion can fail.
\begin{example}\label{eg_counter}
Let the prior distribution be $\mu=\frac14\delta_{u_0}+\frac12\delta_{u_1}+\frac14\delta_{u_2}$, and take $T(u_0)=T(u_2)=0$ and $T(u_1)=1$. Choose the natural parameters $u_0,u_1,u_2$ to correspond to the success probabilities $p_0=\frac14$, $p_1=\frac12$, and $p_2=\frac34$, respectively. Then $\Pi_n^T$ is the posterior probability $\P(\Theta=u_1\vert\mathcal F_n)$, and $T$ is not monotone. We have $\Pi_0^T=\frac12$. If $X_1=1$, the posterior becomes $\mu_{1,1}=\frac18\delta_{u_0}+\frac12\delta_{u_1}+\frac38\delta_{u_2}$, whereas if $X_1=0$, it becomes $\mu_{1,0}=\frac38\delta_{u_0}+\frac12\delta_{u_1}+\frac18\delta_{u_2}$. Thus $\P(\Pi_1^T=\frac12\vert\Pi_0^T=\frac12)=1$. Starting from $\mu_{1,1}$, for example, $X_2=1$ with probability $\frac{9}{16}$, in which case $\Pi_2^T=\frac49$; if $X_2=0$, then $\Pi_2^T=\frac47$. Therefore, the conditional distribution of $\Pi_2^T$ given $\Pi_1^T=\frac12$ is
\[\frac{9}{16}\delta_{\frac49}+\frac{7}{16}\delta_{\frac47}\geq_{\cx}\delta_{\frac12}.\]
Thus the convex order is the reverse of that in Theorem~\ref{thm1}.
\end{example}
The monotonicity of $T$ is used in two places. First, it makes $y\mapsto\Lambda_t(y)$ strictly increasing. Hence, for fixed $t$, the current value $z=\Pi_t^T$ uniquely determines the sufficient statistic $y$ and therefore the posterior $\mu_{t,y}$. Without monotonicity, a single $z$-level can contain distinct values $y_1$ and $y_2$, corresponding to different posterior distributions $\mu_{t,y_1}$ and $\mu_{t,y_2}$. Second, the function $q$ used above may fail to be non-decreasing, in which case Lemma~\ref{lem_frozen} cannot be applied.
\subsection{Exponential families of L\'evy processes}\label{sec_continuous}
The interpolation in \eqref{eq_t} has a genuine continuous-time realization for certain exponential families; see \cite{KS}. We make the construction and the admissible state domain precise before stating the continuous-time version of the two main results.

Recall the one-period observation law $P_u(\ud x)=\e^{ux-B(u)}\nu(\ud x)$. Suppose that $P_{u_0}$ is nondegenerate and infinitely divisible for some $u_0\in N^\circ$. Let $L=(L_r)_{r\geq0}$ be a L\'evy process under $\P_{u_0}$ with $L_1\sim P_{u_0}$, and let $\mathcal F_r^L$ denote its raw natural filtration. For $u\in N^\circ$, define the Esscher law $\P_u$ by
\begin{equation}\label{eq_esscher_family}
\left.
\frac{\ud\P_u}{\ud\P_{u_0}}
\right|_{\mathcal F_r^L}
=
\exp\!\left\{(u-u_0)L_r-r(B(u)-B(u_0))\right\},
\quad r\geq0.
\end{equation}
Let $\Theta$ have prior distribution $\mu$, supported on $S\subset N^\circ$, and, conditionally on $\Theta=u$, let $L$ have law $\P_u$. As before, let $T:S\to\R$ be non-decreasing, integrable under $\mu$, and not $\mu$-a.s.\ constant. For $r\geq 0$, define $\Pi_r^T:=\E[T(\Theta)\vert\mathcal F_r^L]$, and define the regular posterior domain
\begin{equation}\label{eq_regular_domain}
\mathscr D_T:=\operatorname{int}_{[0,\infty)\times\R}\left\{
(r,y):\int_S(1+|T(u)|)(1+u^2+|B(u)|)\e^{uy-rB(u)}\mu(\ud u)<\infty\right\},
\end{equation}
where the interior is taken relative to $[0,\infty)\times\R$. Let $\mathscr Y_r:=\{y:(r,y)\in\mathscr D_T\}$ and $\mathscr Z_r:=\Lambda_r(\mathscr Y_r)$. Since $y\mapsto\Lambda_r(y)$ is strictly increasing, every $z\in\mathscr Z_r$ determines a unique $y(r,z)\in\mathscr Y_r$. Moreover, $\mathscr Y_r\subseteq\mathcal Y_r^{(2)}\subseteq\mathcal Y_r$ and $\mathscr Z_r\subseteq\mathcal Z_r^{(2)}$. Thus this is the same inverse and the same state-indexed posterior $\mu_{r,z}$ defined before. For $h>0$, let $Q_u^h$ denote the law of $L_h$ under $\P_u$. For an increment of length $h$, define 
\begin{align*}
    \ell_h(u,x)&:=\frac{\ud Q_u^h}{\ud Q_{u_0}^h}(x)
=\exp\!\left\{(u-u_0)x-h\bigl(B(u)-B(u_0)\bigr)\right\}\\
\overline Q_{r,z}^h(\ud x)&:=\int_S Q_u^h(\ud x)\mu_{r,z}(\ud u),\quad z\in\mathscr Z_r.
\end{align*}
For $\overline Q_{r,z}^h$-almost every $x$, the posterior mean of $T$ after observing the increment $x$ is
\[M_{r,z}^{(h)}(x):=\frac{
\displaystyle\int_S T(u)\ell_h(u,x)\mu_{r,z}(\ud u)}{\displaystyle\int_S\ell_h(u,x)\mu_{r,z}(\ud u)}=\Lambda_{r+h}(y(r,z)+x).\]
Here the last expression is understood wherever $\Lambda_{r+h}$ is defined; this domain contains $y(r,z)+x$ for $\overline Q_{r,z}^h$-almost every $x$. Set $M_{r,z}^{(h)}(x)=z$ on the exceptional null set. For $A\in\mathcal B(\R)$, define
\begin{equation}\label{eq_ct_kernel}
K_{r,h}^T(z,A):=\int_{\R}\ind_A\!\left(M_{r,z}^{(h)}(x)\right)\overline Q_{r,z}^h(\ud x),\quad z\in\mathscr Z_r,
\end{equation}
and let $\mathcal K_{r,h}^Tf(z):=\int_{\R}f(w)K_{r,h}^T(z,\ud w)$. For each fixed $r$, the inverse map $z\mapsto y(r,z)$ is continuous on $\mathscr Z_r$, and the numerator and denominator in the update formula are jointly measurable in $(z,x)$. Thus \eqref{eq_ct_kernel} is a Borel transition kernel. Its value is independent of the convention on the exceptional null set, so it provides a canonical definition of the transition law even when $\{\Pi_r^T=z\}$ has probability zero.

\begin{proposition}[Esscher--L\'evy posterior kernels]\label{prop_esscher_levy}
Under the preceding construction, $L_r$ is a sufficient statistic for $\Theta$. For $r\geq 0$ and for every Borel set $A\subseteq S$, $\P(\Theta\in A\vert\mathcal F_r^L)=\mu_{r,L_r}(A)$ a.s., and consequently $\Pi_r^T = \Lambda_r(L_r)$ almost surely. Fix $0\leq t<s$ and $h>0$, and let $J\subseteq\bigcap_{r\in[t,s]}\mathscr Z_r$ be a nonempty open interval. Suppose that, for every compact interval $J_0\subset J$, the set $\{(r,y(r,z)):r\in[t,s],\ z\in J_0\}$ has compact closure contained in $\mathscr D_T$, defined in \eqref{eq_regular_domain}. Then:
\begin{enumerate}[(i)]
\item For every $z\in J$, $K_{t,h}^T(z,\cdot)\geq_{\cx}K_{s,h}^T(z,\cdot)$. Equivalently, $\mathcal K_{t,h}^Tf(z)\geq\mathcal K_{s,h}^Tf(z)$ for every convex function $f$ for which both sides are finite.
\item For each $r\in[t,s]$ and every finite convex function
$f:\R\to\R$ satisfying $\mathcal K_{r,h}^T|f|(z)<\infty$ for $z\in J$,
the map $z\mapsto\mathcal K_{r,h}^Tf(z)$ is convex on $J$.
\end{enumerate}
\end{proposition}

\begin{proof}
For $u\in N^\circ$, let $D_r^u$ denote the density on the right-hand side of \eqref{eq_esscher_family}.
The L\'evy exponential-moment identity gives, for $0\leq \ell\leq r$,
\[
\E_{u_0}[D_r^u\vert\mathcal F_\ell^L]
=D_\ell^u\E_{u_0}\!\left[
\e^{(u-u_0)(L_r-L_\ell)-(r-\ell)(B(u)-B(u_0))}
\right]
=D_\ell^u.
\]
Thus the density in \eqref{eq_esscher_family} is a martingale and defines a consistent family $(\P_u)_{u\in N^\circ}$ under which $L$ remains a L\'evy process. At time one, the Esscher formula and the definition of $P_u$ show that $L_1$ has law $P_u$ under $\P_u$.
Hence, under $\P_u$, $L_n$ has the same law as the discrete sufficient statistic $Y_n$ conditional on $\Theta=u$, for every integer $n\geq1$. Consequently, when $h=k$ is an integer, $Q_u^k=P_u^{(k)}$ and the continuous-time kernel \eqref{eq_ct_kernel} agrees with the discrete kernel \eqref{eq_discrete_kernel} on $\mathscr Z_r$.

Bayes' formula and \eqref{eq_t} give the posterior identity stated above. Taking $r=h$ in \eqref{eq_esscher_family} and pushing the resulting measure forward under $L_h$ gives the likelihood $\ell_h$; stationary independent increments make this the likelihood ratio for any future increment of length $h$. Another application of Bayes' formula gives the posterior update above. The Tonelli argument following \eqref{eq_discrete_kernel}, with $(P_u^{(k)},P_{u_*}^{(k)},\ell_k)$ replaced by $(Q_u^h,Q_{u_0}^h,\ell_h)$, shows that the denominator is positive and finite and that the absolute numerator is finite on a common $Q_{u_0}^h$-full set. It also shows that $\overline Q_{r,z}^h$ is equivalent to $Q_{u_0}^h$. The analogue of \eqref{eq_kernel_mean} is $\int_{\R}wK_{r,h}^T(z,\ud w)=\int_ST(u)\mu_{r,z}(\ud u)=z$. From the likelihood ratio above, if $u_1<u_2$, then the likelihood ratio of $Q_{u_2}^h$ to $Q_{u_1}^h$ is increasing in $x$. Thus the future-increment experiment has the same monotone-likelihood-ratio property as a block of discrete observations. All the increment laws $Q_u^h$ are equivalent to $Q_{u_0}^h$. On the common full-measure set where the ratio is defined, $M_{r,z}^{(h)}$ is non-decreasing in $x$. Consequently, each of its superlevel sets agrees $Q_{u_0}^h$-almost everywhere with an upper Borel set. Replacing a superlevel set by this upper set changes neither the kernel nor the probabilities below.

For part~(i), fix $a\in\R$ and write $c_a(w):=(w-a)^+$. At time $s$, choose an upper Borel set $A$ that agrees $Q_{u_0}^h$-almost everywhere with $\{x:M_{s,z}^{(h)}(x)>a\}$, and set $q(u):=Q_u^h(A)$. Apply the argument in the proof of Theorem~\ref{thm1} with a future increment $Y$ whose conditional law is $Q_u^h$. The domain \eqref{eq_regular_domain} supplies the domination needed to differentiate along $r\mapsto\mu_{r,z}$. The preceding likelihood-ratio calculation makes $q$ non-decreasing, while the score $G_r$ is concave and satisfies the two orthogonality relations in \eqref{eq_fix}. Lemma~\ref{lem_frozen} therefore gives
\[
\mathcal K_{t,h}^Tc_a(z)\geq\mathcal K_{s,h}^Tc_a(z).
\]
Together with the common-mean identity above and the call-function characterization \eqref{eq_call_cx}, this proves part (i). When $t=0$, the derivative is understood from the right, and the endpoint follows by continuity.

For part~(ii), fix $r\in[t,s]$ and $a\in\R$. Repeat the construction in the proof of Theorem \ref{thm_convex}, replacing $P_u^{(k)}$ by $Q_u^h$, setting $q_A(u):=Q_u^h(A)$, and replacing $t$ by $r$. The continuous-time analogue of \eqref{eq_frozen_event}, obtained from Bayes' formula, identifies each frozen functional with its event integral and hence gives the envelope. On every compact interval $J_0\subset J$, the compact-closure assumption in the proposition and \eqref{eq_regular_domain} make the derivative estimate \eqref{eq_uniform_lipschitz} uniform over $A$. The envelope is therefore locally Lipschitz on $J$.

Now fix $z_0\in J$ and choose an upper Borel set $A$ that agrees $Q_{u_0}^h$-almost everywhere with $\{x:M_{r,z_0}^{(h)}(x)>a\}$. The monotone-likelihood-ratio property implies that $q_A$ is non-decreasing. If this maximizing event has probability zero or one, its supporting functional is affine. Otherwise, the domain \eqref{eq_regular_domain} permits the two differentiations in Lemma~\ref{lem_2nd_deriv}; the resulting second-derivative density $Q_{z_0}$ is convex and satisfies \eqref{eq_Q}, so Lemma~\ref{lem_frozen} applies. The supporting-function argument in the proof of Theorem~\ref{thm_convex} now proves convexity for call payoffs. The operator preserves affine functions by the common mean identity, the same mixture argument used at the end of that proof extends the conclusion to every $f$ in part~(ii). Thus the map $z\mapsto\mathcal K_{r,h}^Tf(z)$ is convex.
\end{proof}

The Esscher transform also gives an explicit change of the L\'evy triplet. If the L\'evy process under $\P_{u_0}$ has triplet $(b,c,\nu_L)$, then its triplet under $\P_u$ is
\begin{align*}
b^u&=b+c(u-u_0)+\int_{-1\leq x\leq1}x(\e^{(u-u_0)x}-1)\nu_L(\ud x),\\
\nu_L^u(\ud x)&=\e^{(u-u_0)x}\nu_L(\ud x),\quad c^u=c.
\end{align*}
Thus the drift and jump measure change, whereas the diffusion coefficient does not. The framework therefore covers a Brownian motion with unknown drift and fixed variance, a Poisson process with unknown intensity, and a gamma process with an unknown rate parameter, among other examples. It does not cover arbitrary changes to the L\'evy triplet or the general multidimensional state-dependent jump-diffusion models studied in \cite{ET2}. Our setting also differs in that convexity preservation concerns the posterior statistic rather than the observed process.

As a continuous-time example, let $N$ be a Poisson process satisfying $N_t\vert\{\Theta=u\}\sim\operatorname{Poisson}(t\e^u)$. After observing $N$, the posterior is proportional to $\exp\{uN_t-t\e^u\}\mu(\ud u)$. Sequential testing for Poisson processes with a binary prior was studied in \cite{PS}; subsequent work extended this analysis to finitely many simple hypotheses for compound Poisson observations \cite{DPS}. By contrast, our formulation accommodates a general prior on the Poisson intensity and monotone functionals of the unknown parameter, including composite-hypothesis testing with $T(\Theta)=\ind_{\{\Theta>\theta_0\}}$. The formulation also covers compound Poisson processes, but only through a one-parameter family in which the unknown parameter generally changes both the jump intensity and the jump-size distribution.

More generally, let $A:\R_+\to\R_+$ be increasing, with $A(0)=0$, and define the time-changed process $\widetilde N_t:=N_{A(t)}$. Then
\[\widetilde N_t\vert\{\Theta=u\}\sim\operatorname{Poisson}(A(t)\e^u),\]
and its posterior is proportional to $\exp\{u\widetilde N_t-A(t)\e^u\}\mu(\ud u)$. Thus $A(t)$ measures the accumulated information and gives a continuous-time interpretation of the heterogeneous Poisson observations in Remark~\ref{rem_information}. The same idea applies more generally to a deterministic time change of a L\'evy process. The natural comparison is then between intervals that contribute the same amount of information.

Consider a Brownian motion with unknown drift, $Y_t=\Theta t+W_t$, and let $\mathcal F_t^Y$ be its observation filtration. Here $B(u)=u^2/2$. Define $\hat\Theta_t:=\E[\Theta\vert\mathcal F_t^Y]$ and the innovation process $\hat W_t:=Y_t-\int_0^t\hat\Theta_r \ud r$. It\^o's formula gives
\[\ud\Pi_t^T=\Cov(T(\Theta),\Theta\vert\mathcal F_t^Y) \ud\hat W_t.\]
Indeed, the diffusion coefficient is $\partial_y\Lambda_t(Y_t)=\Cov(T(\Theta),\Theta\vert\mathcal F_t^Y)$, whereas the drift vanishes because $\Pi^T$ is a martingale. By Corollary~\ref{cor_cov_monot},
the diffusion coefficient of the $\Pi^T$ process decreases in time along each level curve. Two applications are immediate. First, if $T(u)=u$, then $\Cov_{\mu_{t,z}}(T(\Theta),\Theta)=\Var_{\mu_{t,z}}(\Theta)$ and $\Var_{\mu_{t,z}}(\Theta)\geq\Var_{\mu_{s,z}}(\Theta)$ for $t<s$. This recovers the monotonicity of the diffusion coefficient in Proposition~2.5 of \cite{EKV} and Proposition~3.6 of \cite{EV2}. Our result does not, however, fully recover the time-monotonicity result of \cite{EKV}, because convexity of the relevant payoff functions has not been established. Second, if $T(u)=\ind_{(\theta_0,\infty)}(u)$, define $\sigma(t,\pi):=\Cov_{\mu_{t,\pi}}(\ind_{(\theta_0,\infty)}(\Theta),\Theta)$. Then $\sigma(t,\pi)\geq\sigma(s,\pi)$ for $t<s$, recovering Corollary~3.10 of \cite{EV}. Together with Theorem~\ref{thm_convex}, this can be used, under the assumptions of the relevant stopping problem, to establish monotonicity of the value function and the stopping boundaries.

\begin{example}\label{eg1_BM}
Consider a Brownian motion with unknown drift, $Y_t=\Theta t+W_t$. Fix a time window $h>0$ and a threshold $a>0$. At time $t$, consider forecasting the probability that the observed process rises by at least $a$ during the next interval of length $h$:
\[\Pi_t:=\P\left(\sup_{0\leq r\leq h}(Y_{t+r}-Y_t)\geq a\vert\mathcal F_t^Y\right).\]
Conditional on $\Theta=u$, the future increments satisfy $Y_{t+r}-Y_t\stackrel{d}{=}ur+W_r$ for $0\leq r\leq h$. Define $T(u):=\P\left(\sup_{0\leq r\leq h}(ur+W_r)\geq a\right)$. Then $\Pi_t=\E[T(\Theta)\vert\mathcal F_t^Y]$. To see that $T$ is non-decreasing, observe that if $u_1\leq u_2$, then $u_1r+W_r\leq u_2r+W_r$ for every $r\geq0$ and every sample path. Therefore,
\[
\left\{\sup_{0\leq r\leq h}(u_1r+W_r)\geq a\right\}
\subseteq
\left\{\sup_{0\leq r\leq h}(u_2r+W_r)\geq a\right\}.
\]
In fact, the reflection principle gives
\[T(u)=\Phi\left(u\sqrt h-\frac{a}{\sqrt h}\right)+\e^{2au}\Phi\left(-u\sqrt h-\frac{a}{\sqrt h}\right).\]
Moreover, $\Pi$ satisfies $\ud\Pi_t=\sigma(t,\Pi_t) \ud\hat W_t$, where $\sigma(t,z)=\Cov_{\mu_{t,z}}(T(\Theta),\Theta)$ and $\sigma(t,z)\geq\sigma(s,z)$ for $t\leq s$. Thus, if two forecasts currently assign the same crossing probability $z$, the forecast based on more accumulated observations has a less dispersed subsequent update.
\end{example}

\bibliographystyle{amsplain}
\bibliography{references}
\end{document}